\documentclass{ws-ijbc}
\usepackage{ws-rotating}     % used only when sideways tables/figures are used
\usepackage{graphicx}
\usepackage{epstopdf}
\begin{document}

\catchline{}{}{}{}{} % Publisher's Area please ignore

\markboth{Y. Wu}{Weak bi-center and critical period bifurcations of a quintic system}

\title{Weak bi-center and critical period bifurcations of a $Z_2$-Equivariant quintic system\footnote{This research is partially supported by National Natural Science Foundation of China (No. 11101126).}}

\author{Yusen Wu}

\address{School of Statistics, Qufu Normal University\\
Qufu 273165, Shandong, PR China\\
wuyusen621@126.com}

\maketitle
\begin{history} \received{(Apr. 7, 2020)}
\end{history}

\begin{abstract}
With the help of computer algebra system-\textsc{Mathematica}, this paper considers the weak center problem and local critical periods for bi-center of a $Z_2$-Equivariant quintic system with eight parameters. The order of weak bi-center is identified and the exact maximum number of bifurcation of critical periods generated from the bi-center is given via the combination of symbolic calculation and numerical analysis.
\end{abstract}


%\begin{multicols}{2}
\section{Introduction}

For the following center-focus type planar differential system
\begin{equation}\label{1.1}
\frac{dx}{dt}=-y+h.o.t.,\ \frac{dy}{dt}=x+h.o.t.,
\end{equation}
the center and isochronous center problems as well as bifurcation of limit cycles have attracted much attention from mathematicians. However, another two important topics worthy of investigation for system \eqref{1.1} are identifying the order of weak center and the number of bifurcation of critical periods. A global study of the number of critical points of the period is a very difficult question. However, a simpler version is the local problem of the number of critical periods which can appear by perturbation of a system in the neighborhood of a center. This question is attacked by computing the Taylor series of the period function in the neighborhood of the center and further by determining the order of its first non-constant term. The involved computations are purely algorithmic. When it is performed on a polynomial family of vector fields the coefficients of the period function are polynomials in the coefficients of the system.

Chicone and Jacobs [Chicone \& Jacobs, 1989] introduced the notion of bifurcation of local critical periods by analogy with the method of Bautin [Bautin, 1954] and proved that at most two critical period bifurcations can occur in a quadratic system. Later on, Lin and Li [Lin \& Li, 1991] proposed a complex method to investigate weak centers and local critical periods and solved the bifurcation of local critical periods for the cubic complex system without quadratic terms, which was also studied in [Romanovski \& Han, 2003; Rousseau \& Toni, 1993]. Rousseau and Toni studied the local bifurcation of critical periods of periodic orbits in the neighborhood of a nondegenerate center of a vector field with a homogeneous nonlinearity of the third degree and the reduced Kukles system in [Rousseau \& Toni, 1993] and [Rousseau \& Toni, 1997], respectively. References [Chicone \& Jacobs, 1989] and [Cherkas et al, 1997] obtained the result that $k$ critical points can bifurcate from a weak center of order $k$. Zhang et al. [Zhang et al, 2000] discussed weak center conditions and bifurcation of critical periods in a special reversible cubic systems. Du [Du, 2004] studied local bifurcations of critical periods in the neighborhood of a nondegenerate center of a Li\'{e}nard system. Gasull and Zhao [Gasull \& Zhao, 2008] concerned with the study of the number of critical periods of perturbed isochronous centers. Cima et al. [Cima et al, 2008] got some lower bounds for the number of critical periods of families of centers which are perturbations of the linear one. Chen and Zhang [Chen \& Zhang, 2009] decomposed algebraic sets, stratum by stratum, into a union of constructible sets with Sylvester resultants, so as to simplify the procedure of elimination. Yu et al. [Yu et al, 2010] considered the critical periods of third-order planar Hamiltonian systems. Li et al. [Li et al, 2018] studied bi-center problem and bifurcation of limit cycles from nilpotent singular points in $Z_2$-equivariant cubic vector fields. Li et al. [Li et al, 2020a] studied complex isochronous center problem for cubic complex planar vector fields, which are assumed to be $Z_2$-equivariant with two symmetric centers. Li et al. [Li et al, 2020b] investigated complete integrability and linearizability of cubic $Z_2$ systems with two non-resonant and elementary singular points.

We remind the differential system with $Z_2$-equivariant symmetry, i.e., systems with unchanged phase portraits after a rotation on the angle $\pi$ around a point $P$. Assume that $A$ and $B$ are singular points of a differential system which is $Z_2$-equivariant with respect to the middle point of the line segment $AB$. We say that such system has a \textit{bi-center} at points $A$ and $B$ if both $A$ and $B$ are singular points of the center type. Du et al. [Du et al, 2013] devoted to study a center problem and a weak center problem for cubic systems in $Z_4$-equivariant vector fields. Chen et al. [Chen et al, 2014] considered the weak center conditions and local critical periods for a $Z_2$-equivariant cubic system with eleven center conditions at the bi-center. Romanovski et al. [Romanovski et al, 2017] investigated the existence of a bi-center for a subfamily of a planar $Z_2$-equivariant differential quintic system of the form
\begin{equation}\label{1.2}
\frac{dx}{dt}=X_1(x,y)+X_5(x,y),\ \frac{dy}{dt}=Y_1(x,y)+Y_5(x,y),
\end{equation}
where $X_i(x,y),Y_i(x,y)(i=1,5)$ are homogeneous polynomials of degree $i$ in the variables $x$ and $y$ and \eqref{1.2} has two weak foci or centers at the points $(\pm1,0)$. Four families of system \eqref{1.2} possessing bi-centers were found and the authors also have shown that they are not isochronous. The reason for choosing homogeneous polynomials of degree five in \eqref{1.2} (rather than polynomials of degree four) is to assure the existence of $Z_2$-equivariant symmetry with respect to origin, which can appear only if the polynomials defining the system have just odd degree monomials. So, if replacing in \eqref{1.2} $X_5(x,y)$ and $Y_5(x,y)$ with homogeneous polynomials of degree four, then the system cannot have a bi-center at $(\pm1,0)$, and if we add homogeneous perturbations of degree three, then the study becomes computationally unfeasible.

The main propose of this paper is to continue the investigation of system \eqref{1.2}. Specifically speaking, we study the conditions on the parameters for such system to have a weak bi-center of maximum order and the most bifurcation of critical periods.

The paper is organized as follows. In Section 2 we show the computational method of period constants. In Section 3 we state and prove our main results. In the last section, we give a short conclusion.

\section{Preliminary knowledge}

Consider an autonomous two-dimensional system of the form
\begin{equation}\label{2.1}\begin{array}{l}
\frac{dx}{dt}=-y+\sum\limits_{k=2}^\infty X_k(x,y,\lambda),\\
\frac{dy}{dt}=x+\sum\limits_{k=2}^\infty Y_k(x,y,\lambda),
\end{array}\end{equation}
where $X_k(x,y,\lambda),Y_k(x,y,\lambda)$ are homogeneous
polynomials of degree $k$ of $x,y$ and
$\lambda\in\Lambda\subset\mathbb{R}^s$ is a vector of parameters.
Assume $X_k(x,y,\lambda),Y_k(x,y,\lambda)$ are analytic in a
neighborhood of the origin which is a center type singularity of
system \eqref{2.1}. Under the polar coordinates
$x=r\cos\theta,y=r\sin\theta$, system \eqref{2.1} takes the form
\begin{equation}\label{2.2}
\frac{dr}{d\theta}=\frac{\sum\limits_{k=2}^\infty
r^k\varphi_{k+1}(\theta)}{1+\sum\limits_{k=2}^\infty
r^{k-1}\psi_{k+1}(\theta)},
\end{equation}
where
\begin{equation}\begin{array}{l}
\varphi_{k+1}(\theta)=\cos\theta X_k(\cos\theta,\sin\theta)+\sin\theta Y_k(\cos\theta,\sin\theta),\\
\psi_{k+1}(\theta)=\cos\theta Y_k(\cos\theta,\sin\theta)-\sin\theta
X_k(\cos\theta,\sin\theta),
\end{array}\end{equation}
$k=2,3,\cdots$ and
\begin{equation}
\frac{d\theta}{dt}=1+\sum\limits_{k=2}^\infty
r^{k-1}\psi_{k+1}(\theta).
\end{equation}
Let $r(\theta,h)$ be the solution of system \eqref{2.2} associated
with the initial condition $r|_{\theta=0}=h$ (it corresponds to the
initial point $(h,0)$ in the rectangular coordinate system). For a
sufficiently small real constant $h$ the period function is defined
by
\begin{equation}
P(h,\lambda)=\int_0^{2\pi}\frac{d\theta}{1+\sum\limits_{k=2}^\infty
r^{k-1}(\theta,h)\psi_{k+1}(\theta)}=2\pi+\sum\limits_{k=1}^\infty
T_kh^k.
\end{equation}
It is known from [Romanovski \& Shafer, 2009] that the period
constants can be also written in the form
\begin{equation}
P(h,\lambda)=2\pi+\sum\limits_{k=1}^\infty p_{2k}(\lambda)h^{2k}.
\end{equation}
The coefficient $p_{2k}$ in the above expression is call the $k$-th
period constant at the origin of the system.

\begin{definition}
Let $\phi(h,\lambda):=P(h,\lambda)-2\pi$. If there exists $k\in\mathbb{N},\lambda_*\in\mathbb{R}^n$ such that
\begin{equation}
\phi(0,\lambda_*)=\phi'(0,\lambda_*)=\cdots=\phi^{(2k+1)}(0,\lambda_*)=0,\phi^{(2k+2)}(0,\lambda_*)\neq0,
\end{equation}
or equivalently
\begin{equation}
p_2(\lambda_*)=p_4(\lambda_*)=\cdots=p_{2k}(\lambda_*)=0,p_{(2k+2)}(\lambda_*)\neq0,
\end{equation}
the origin of system \eqref{2.1} is called a $k$-order weak center at the parameter $\lambda_*$. If $k=0$, the origin is called a strong center. If $p_{2k}(\lambda_*)=0$ for all positive integer $k$, then the origin is called an isochronous center.
\end{definition}

\begin{definition}
Let the origin be a weak or isochronous center of system \eqref{2.1} corresponding to the parameter $\lambda_*\in\Lambda$. It is said that $k$ local critical periods bifurcate from the origin if there is $\varepsilon_0>0$ such that for every $0<\varepsilon<\varepsilon_0$ and every sufficiently small neighborhood $W$ of $\lambda_*$, there is a $\lambda_1\in W$ such that $P'(h,\lambda_1)=0$ has $k$ solutions in $U=(0,\varepsilon)$.
\end{definition}

By means of transformation
\begin{equation}
z=x+iy,\ w=x-iy,\ T=it,\ i=\sqrt{-1},
\end{equation}
system \eqref{2.1} can be transformed into the following complex
system
\begin{equation}\label{2.8}\begin{array}{l}
\frac{dz}{dT}=z+\sum\limits_{k=2}^{\infty}Z_k(z,w,\lambda)=Z(z,w,\lambda),\\
\frac{dw}{dT}=-w-\sum\limits_{k=2}^{\infty}W_k(z,w,\lambda)=-W(z,w,\lambda),
\end{array}\end{equation}
with
\begin{equation}
Z_k(z,w,\lambda)=\sum\limits_{\alpha+\beta=k}a_{\alpha\beta}(\lambda)z^\alpha
w^\beta,\
W_k(z,w,\lambda)=\sum\limits_{\alpha+\beta=k}b_{\alpha\beta}(\lambda)w^\alpha
z^\beta.
\end{equation}
It is obvious that the coefficients of system \eqref{2.8} satisfy
the conjugate condition, i.e.,
\begin{equation}
\overline{a_{\alpha\beta}}=b_{\alpha\beta},\ \alpha\geq0,\
\beta\geq0,\ \alpha+\beta\geq2.
\end{equation}
We call that systems \eqref{2.1} and \eqref{2.8} are concomitant.

\begin{lemma}
([Amelkin et al, 1982]) For system \eqref{2.8}, we can derive
uniquely the following formal series:
\begin{equation}
\xi=z+\sum\limits_{k+j=2}^\infty c_{kj}z^kw^j,\
\eta=w+\sum\limits_{k+j=2}^\infty d_{kj}w^kz^j,
\end{equation}
where $c_{k+1,k}=d_{k+1,k}=0,k=1,2,\cdots$, such that
\begin{equation}
\frac{d\xi}{dT}=\xi+\sum\limits_{j=1}^{\infty}p_j\xi^{j+1}\eta^j,\
\frac{d\eta}{dT}=-\eta-\sum\limits_{j=1}^{\infty}q_j\eta^{j+1}\xi^j.
\end{equation}
\end{lemma}

\begin{definition}
The quantity $\tau_k=p_k+q_k,k=1,2,\cdots$ is called the $k$-th
\textit{complex period constant} at the origin of system
\eqref{2.8}. A complex center is called a \textit{complex
isochronous center} if all $\tau_k$ vanish.
\end{definition}

The method to compute $\tau_k$ is shown in the following theorems:

\begin{theorem}\label{t2.1}
([Liu \& Huang, 2003]) For system \eqref{2.8}, we can derive uniquely
the following formal series:
\begin{equation}
f(z,w)=z+\sum\limits_{k+j=2}^{\infty}c'_{kj}z^kw^j,\
g(z,w)=w+\sum\limits_{k+j=2}^{\infty}d'_{kj}w^kz^j,
\end{equation}
where $c'_{k+1,k}=d'_{k+1,k}=0,k=1,2,\cdots,$ such that
\begin{equation}\frac{df}{dT}=f(z,w)+\sum\limits_{j=1}^{\infty}p'_jz^{j+1}w^j,\ \frac{dg}{dT}=-g(z,w)-\sum\limits_{j=1}^{\infty}q'_jw^{j+1}z^j,
\end{equation}
and when $k-j-1\neq0,c'_{kj}$ and $d'_{kj}$ are determined by the
following recursive formulae:
\begin{equation}\label{2.15}\begin{array}{l}
c'_{kj}={1\over{j+1-k}}\sum\limits_{\alpha+\beta=3}^{k+j+1}[(k-\alpha+1)a_{\alpha,\beta-1}-(j-\beta+1)b_{\beta,\alpha-1}]c'_{k-\alpha+1,j-\beta+1},\\
d'_{kj}={1\over{j+1-k}}\sum\limits_{\alpha+\beta=3}^{k+j+1}[(k-\alpha+1)b_{\alpha,\beta-1}-(j-\beta+1)a_{\beta,\alpha-1}]d'_{k-\alpha+1,j-\beta+1},
\end{array}\end{equation}
and for any positive integer $j,p'_j$ and $q'_j$ are determined by
the following recursive formulae:
\begin{equation}\label{2.16}\begin{array}{l}
p'_j=\sum\limits_{\alpha+\beta=3}^{2j+2}[(j-\alpha+2)a_{\alpha,\beta-1}-(j-\beta+1)b_{\beta,\alpha-1}]c'_{j-\alpha+2,j-\beta+1},\\
q'_j=\sum\limits_{\alpha+\beta=3}^{2j+2}[(j-\alpha+2)b_{\alpha,\beta-1}-(j-\beta+1)a_{\beta,\alpha-1}]d'_{j-\alpha+2,j-\beta+1}.
\end{array}\end{equation}
In expressions \eqref{2.15} and \eqref{2.16}, we have let
$c'_{10}=d'_{10}=1,c'_{01}=d'_{01}=0,$ and if $\alpha<0$ or
$\beta<0$, let
$a_{\alpha\beta}=b_{\alpha\beta}=c'_{\alpha\beta}=d'_{\alpha\beta}=0.$
\end{theorem}

The relations between $p_j,q_j$ and $p'_j,q'_j\;(j=1,2,\cdots)$ are
as follows:

\begin{theorem}
([Liu \& Huang, 2003]) Let $p_0=q_0=p'_0=q'_0=0$. If there exists a
positive integer $k$, such that
\begin{equation}
p_0=q_0=p_1=q_1=\cdots=p_{k-1}=q_{k-1}=0,
\end{equation}
then
\begin{equation}
p'_0=q'_0=p'_1=q'_1=p'_{k-1}=q'_{k-1}=0,p_k=p'_k,q_k=q'_k
\end{equation} per contra, it holds as well.
\end{theorem}

In [Liu et al, 2008] it is shown that the first nonzero complex
period constant $\tau_k$ at the origin of system \eqref{2.8} and the
first nonzero period constant $p_{2k}$ of system \eqref{2.1} satisfy
the relation
\begin{equation}
p_{2k}=-\pi\tau_k.
\end{equation}

\section{Weak centers and local critical periods}

In [Romanovski et al, 2017] the authors investigated the existence of bi-centers in family \eqref{1.2}. After a change of coordinates, the following normal form for system \eqref{1.2} was found
\begin{equation}\begin{array}{l}\label{3.1}
\frac{dx}{dt}=-(a_1+1)y+a_1x^4y+a_2x^3y^2+a_3x^2y^3+a_4xy^4+a_5y^5,\\
\frac{dy}{dt}=-\frac{1}{4}x-a_6y+\frac{1}{4}x^5+a_6x^4y+a_7x^3y^2+a_8x^2y^3+a_9xy^4+a_{10}y^5,
\end{array}\end{equation}
where $a_i\in\mathbb{R},i=1,\cdots,10$. Since system  \eqref{3.1} is $Z_2$-equivariant, the existence of a weak bi-center at the points $(\pm1,0)$ follows from the existence of a weak bi-center at the point $(1,0)$.
Moreover, to compute the period constants for system \eqref{3.1} it is necessary to move the singular point $(1,0)$ to the origin. Applying the transformation
\begin{equation}
\tilde{x}=x-1,\ \tilde{y}=y,
\end{equation}
system \eqref{3.1} is changed into the system
\begin{equation}\begin{array}{l}
\frac{dx}{dt}=-y+4a_1xy+a_2y^2+6a_1x^2y+3a_2xy^2+a_3y^3+4a_1x^3y+3a_2x^2y^2+2a_3xy^3\\
\ \ \ \ \ \ \ +a_4y^4+a_1x^4y+a_2x^3y^2+a_3x^2y^3+a_4xy^4+a_5y^5,\\
\frac{dy}{dt}=x+\frac{5}{2}x^2+4a_6xy+a_7y^2+\frac{5}{2}x^3+6a_6x^2y+3a_7xy^2+a_8y^3+\frac{5}{4}x^4+4a_6x^3y\\
\ \ \ \ \ \ \ +3a_7x^2y^2+2a_8xy^3+a_9y^4+\frac{1}{4}x^5+a_6x^4y+a_7x^3y^2+a_8x^2y^3+a_9xy^4+a_{10}y^5,
\end{array}\end{equation}
where we still write $x$ and $y$ instead of $\tilde{x}$ and $\tilde{y}$.

Unfortunately, because system \eqref{3.1} admits ten parameters the computations become unfeasible for the general case. So Romanovski et al. [Romanovski et al, 2017]
restrict their study to a subcase of system \eqref{3.1} with $a_1=-1,a_5=0$ which possesses the $y$-axis as an invariant curve. Thus, they look for necessary and sufficient
conditions for the system
\begin{equation}\label{3.4}\begin{array}{l}
\frac{dx}{dt}=-x^4y+a_2x^3y^2+a_3x^2y^3+a_4xy^4,\\
\frac{dy}{dt}=-\frac{1}{4}x-a_6y+\frac{1}{4}x^5+a_6x^4y+a_7x^3y^2+a_8x^2y^3+a_9xy^4+a_{10}y^5,
\end{array}\end{equation}
to have a bi-center at the points $(\pm1,0)$, or, equivalently, for the system
\begin{equation}\label{3.5}\begin{array}{l}
\frac{dx}{dt}=-y-4xy+a_2y^2-6x^2y+3a_2xy^2+a_3y^3-4x^3y+3a_2x^2y^2+2a_3xy^3\\
\ \ \ \ \ \ \ +a_4y^4-x^4y+a_2x^3y^2+a_3x^2y^3+a_4xy^4,\\
\frac{dy}{dt}=x+\frac{5}{2}x^2+4a_6xy+a_7y^2+\frac{5}{2}x^3+6a_6x^2y+3a_7xy^2+a_8y^3+\frac{5}{4}x^4+4a_6x^3y\\
\ \ \ \ \ \ \ +3a_7x^2y^2+2a_8xy^3+a_9y^4+\frac{1}{4}x^5+a_6x^4y+a_7x^3y^2+a_8x^2y^3+a_9xy^4+a_{10}y^5,
\end{array}\end{equation}
to have a center at the origin. Applying the complexification
\begin{equation}
z=x+iy,\ w=x-iy,\ T=it,\ i=\sqrt{-1},
\end{equation}
system \eqref{3.5} becomes its concomitant complex system
\begin{equation}\label{3.7}\begin{array}{l}
\frac{dz}{dT}=z+\sum\limits_{k+j=2}^5a_{kj}z^kw^j,\\
\frac{dw}{dT}=-\left(w+\sum\limits_{k+j=2}^5b_{kj}w^kz^j\right),
\end{array}\end{equation}
where
\begin{equation}\label{3.8}\begin{array}{l}
a_{20}=\frac{1}{8}(13+2ia_2-8ia_6-2a_7),\\
a_{11}=\frac{1}{4}(5-2ia_2+2a_7),\\
a_{02}=\frac{1}{8}i(3i+2a_2+8a_6+2ia_7),\\
a_{30}=\frac{1}{16}(17+6ia_2+2a_3-12ia_6-6a_7+2ia_8),\\
a_{21}=\frac{3}{16}(9-2ia_2-2a_3-4ia_6+2a_7-2ia_8),\\
a_{12}=\frac{3}{16}(1-2ia_2+2a_3+4ia_6+2a_7+2ia_8),\\
a_{03}=\frac{1}{16}i(7i+6a_2+2ia_3+12a_6+6ia_7-2a_8),\\
a_{40}=\frac{1}{64}(21+12ia_2+8a_3-4ia_4-16ia_6-12a_7+8ia_8+4a_9),\\
a_{31}=\frac{1}{16}(13-4a_3+4ia_4-8ia_6-4ia_8-4a_9),\\
a_{22}=\frac{3}{32}(5-4ia_2-4ia_4+4a_7+4a_9),\\
a_{13}=\frac{1}{16}(-3+4a_3+4ia_4+8ia_6+4ia_8-4a_9),\\
a_{04}=\frac{1}{64}(-11+12ia_2-8a_3-4ia_4+16ia_6-12a_7-8ia_8+4a_9),\\
a_{50}=\frac{1}{128}(5-4ia_{10}+4ia_2+4a_3-4ia_4-4ia_6-4a_7+4ia_8+4a_9),\\
a_{41}=\frac{1}{128}(17+20ia_{10}+4ia_2-4a_3+12ia_4-12ia_6-4a_7-4ia_8-12a_9),\\
a_{32}=\frac{1}{64}(9-20ia_{10}-4ia_2-4a_3-4ia_4-4ia_6+4a_7-4ia_8+4a_9),\\
a_{23}=\frac{1}{64}(1+20ia_{10}-4ia_2+4a_3-4ia_4+4ia_6+4a_7+4ia_8+4a_9),\\
a_{14}=\frac{1}{128}(-7-20ia_{10}+4ia_2+4a_3+12ia_4+12ia_6-4a_7+4ia_8-12a_9),\\
a_{05}=\frac{1}{128}(-3+4ia_{10}+4ia_2-4a_3-4ia_4+4ia_6-4a_7-4ia_8+4a_9)\\
b_{kj}=\overline{a_{kj}},\ k\geq0,\ j\geq0,\ k+j=2,3,4,5.
\end{array}\end{equation}

Denote that $\lambda=(a_2,a_3,a_4,a_6,a_7,a_8,a_9,a_{10})\in\mathbb{R}^8$. The following four necessary and sufficient conditions for the existence of a bi-center at the points $(\pm1,0)$ for the $Z_2$-equivariant system \eqref{3.4} are given in [Romanovski, 2017]:

$\Lambda_1=\left\{\lambda\in\mathbb{R}^8\left|a_6=0,a_8=\frac{1}{3}a_2(1-2a_7),a_9=\frac{1}{2}a_3(1-a_7),a_{10}=\frac{1}{5}a_4(3-2a_7)\right.\right\}$;

$\Lambda_2=\left\{\lambda\in\mathbb{R}^8\left|a_2=-4a_6,a_4=4a_3a_6,a_8=4a_6a_7,a_{10}=4a_6a_9\right.\right\}$;

$\Lambda_3=\left\{\lambda\in\mathbb{R}^8\left|
\begin{array}{l}
a_4=4a_6(a_3-4a_2a_6-16a_6^2),a_8=\frac{1}{3}(a_2+4a_6-2a_2a_7+4a_6a_7),\\
a_9=\frac{1}{6}(3a_3-4a_2a_6-16a_6^2-3a_3a_7-4a_2a_6a_7-16a_6^2a_7),\\
a_{10}=2a_6(-a_3+4a_2a_6+16a_6^2)(-1+a_7)
\end{array}\right.\right\}$;

$\Lambda_4=\left\{\lambda\in\mathbb{R}^8\left|a_7=-1,a_8=a_2,a_9=a_3,a_{10}=a_4\right.\right\}$.

Now we discuss the order of weak bi-center and bifurcation of critical periods when the origin of system \eqref{3.7} is a center. The next computational step is to compute the complex period constants at the origin of system \eqref{3.7}. For this purpose we use the procedure described in Section 2 running with the computer algebra system \textsc{Mathematica}. Since a bi-center is prior to be a weak bi-center, to make the computations easier we identify the order of weak center and investigate the existence of bifurcation of critical periods using the four bi-center conditions given above. Thus, our proof is split in four cases corresponding to these four conditions.

\subsection{Center of type $\Lambda_1$}

Substituting $\Lambda_1$ into formulae \eqref{2.15}-\eqref{2.16} of Theorem \ref{t2.1}, we compute the first four complex period constants at the origin of system \eqref{3.7} as follows:
\begin{equation}\label{3.9}\begin{array}{l}
\tau_1^{(1)}=\frac{1}{12}(-48-10a_2^2-9a_3-36a_7-4a_7^2),\\
\tau_2^{(1)}=\frac{1}{108}(-1224+840a_2^2+140a_2^4-189a_2a_4-1620a_7+630a_2^2a_7-816a_7^2+70a_2^2a_7^2\\
\ \ \ \ \ \ \ \ \ \ \ \ \ \ -180a_7^3-16a_7^4),\\
\tau_3^{(1)}=\frac{1}{103680}(-363672-2421720a_2^2-6039180a_2^4-779800a_2^6+1836513a_2a_4+997290a_2^3a_4\\
\ \ \ \ \ \ \ \ \ \ \ \ \ \ \ \ \ \ -81648a_4^2-531900a_7-4219110a_2^2a_7-4458300a_2^4a_7+1281420a_2a_4a_7\\
\ \ \ \ \ \ \ \ \ \ \ \ \ \ \ \ \ \ -316512a_7^2-2025030a_2^2a_7^2-539980a_2^4a_7^2+202608a_2a_4a_7^2-76860a_7^3-305760a_2^2a_7^3),\\
\tau_4^{(1)}=\frac{1}{2573329305600}f_1(a_3,a_7,a_2,a_4).
\end{array}\end{equation}
In the above expression of $\tau_k^{(1)}$s, we have already let $\tau_1^{(1)}=\cdots=\tau_{k-1}^{(1)}=0,k=2,3,4$.

\begin{theorem}
Denote $\gamma^{(1)}=(a_3,a_7,a_2,a_4)\in\mathbb{R}^4$. For system \eqref{3.4}, if $\lambda\in\Lambda_1$ the bi-center is a weak center of at most 3. Moreover, it is a the third order weak center if and only if

$\Gamma_1=\left\{\gamma^{(1)}\in\mathbb{R}^4\left|\tau_1^{(1)}=\tau_2^{(1)}=\tau_3^{(1)}=0\right.\right\}\backslash\{\gamma_1^{(1)},\gamma_2^{(1)},\gamma_3^{(1)},\gamma_4^{(1)}\}$.

Since the expressions of $f_1(a_3,a_7,a_2,a_4)$ and $\gamma_i^{(1)},i=1,2,3,4$ are tedious, we do not give here their expressions explicitly.
\end{theorem}

\begin{proof}
The \textsc{Mathematica} routine \texttt{Resultant[poly$_1$,poly$_2$,var]} gives the resultant of the polynomials \texttt{poly}$_1$ and \texttt{poly}$_2$ with respect to the variable \texttt{var}. From the algebraic theory, \texttt{Resultant[poly$_1$,poly$_2$,var]}$=0$ is a necessary
condition for \texttt{poly}$_1$=\texttt{poly}$_2$$=0$. Denote by $N_k\in\mathbb{N}$ a natural number of $k$ digits.

First of all, computing the resultant of $\tau_1^{(1)},\tau_2^{(1)},\tau_3^{(1)},\tau_4^{(1)}$ with respect to $a_3$, we have
\begin{equation}\begin{array}{l}
R_{1,2}^{(1)}=\texttt{Resultant[}\tau_1^{(1)},\tau_2^{(1)},a_3\texttt{]}=\frac{1}{108}F_{1,2}^{(1)}(a_7,a_2,a_4),\\
R_{1,3}^{(1)}=\texttt{Resultant[}\tau_1^{(1)},\tau_3^{(1)},a_3\texttt{]}=\frac{1}{103680}F_{1,3}^{(1)}(a_7,a_2,a_4),\\
R_{1,4}^{(1)}=\texttt{Resultant[}\tau_1^{(1)},\tau_4^{(1)},a_3\texttt{]}=\frac{1}{2573329305600}F_{1,4}^{(1)}(a_7,a_2,a_4).
\end{array}\end{equation}
Next, computing the resultant of $R_{1,2}^{(1)},R_{1,3}^{(1)},R_{1,4}^{(1)}$ with respect to $a_7$, we have
\begin{equation}\begin{array}{l}
R_{12,13}^{(1)}=\texttt{Resultant[}R_{1,2}^{(1)},R_{1,3}^{(1)},a_7\texttt{]}=\frac{1}{6855297075118080000}F_{12,13}^{(1)}(a_2,a_4),\\
R_{12,14}^{(1)}=\texttt{Resultant[}R_{1,2}^{(1)},R_{1,4}^{(1)},a_7\texttt{]}=\frac{F_{12,14}^{(1)}(a_2,a_4)}{1200135732896571944053984339414625958651494400000000}.
\end{array}\end{equation}
Finally, computing the resultant of $R_{12,13}^{(1)},R_{12,14}^{(1)}$ with respect to $a_2$, we have
\begin{equation}
R_{1213,1214}^{(1)}=\texttt{Resultant[}R_{12,13}^{(1)},R_{12,14}^{(1)},a_2\texttt{]}=\frac{N_{109}}{N_{1420}}F_{1213,1214}^{(1)}(a_4)G_{1213,1214}^{(1)}(a_4),
\end{equation}
where $F_{1213,1214}^{(1)}(a_4)$ and $G_{1213,1214}^{(1)}(a_4)$ are unary polynomials in $a_4$ of degree 64 and 192, respectively.

Performing the \textsc{Mathematica} routine $\texttt{NSolve[}R_{1213,1214}^{(1)}==0,a_4\texttt{]}$ gives only complex solutions. However, all parameters of system \eqref{3.4} are real, so we get the conclusion that $\tau_1^{(1)},\tau_2^{(1)},\tau_3^{(1)},\tau_4^{(1)}$ have no common real root, which implies that the bi-center of system \eqref{3.4} is at most the third order weak center.

Computing the determinant of Jacobian matrix in this case, we get
\begin{equation}\label{3.16}
\text{det}\left[\frac{\partial(\tau_1^{(1)},\tau_2^{(1)},\tau_3^{(1)})}{\partial(a_3,a_7,a_2)}\right]=\frac{7}{3732480}F(a_7,a_2,a_4).
\end{equation}
Performing the \textsc{Mathematica} routine
\begin{equation}\label{3.17}
\texttt{NSolve[}\{\tau_1^{(1)}==0,\tau_2^{(1)}==0,\tau_3^{(1)}==0,F(a_7,a_2,a_4)==0\},\{a_3,a_7,a_2,a_4\},\texttt{Reals]}
\end{equation}
gives the four real solutions $\gamma_k^{(1)},k=1,2,3,4$.
\end{proof}

\begin{theorem}
For system \eqref{3.4}, if $\lambda\in\Lambda_1$, the maximum number of bifurcation of critical periods is 3, and there are exactly 3 bifurcation of critical periods after a suitable perturbation.
\end{theorem}

\begin{proof}
By checking the Jacobian matrix for this case, we obtain from \eqref{3.17} that
\begin{equation}
\text{det}\left.\left[\frac{\partial(\tau_1^{(1)},\tau_2^{(1)},\tau_3^{(1)})}{\partial(a_3,a_7,a_2)}\right]\right|_{\Gamma_1}\neq0.
\end{equation}
For example, setting $a_4=a_4^*=0$ and performing the \textsc{Mathematica} routine
\begin{equation}
\texttt{NSolve[}\{\tau_1^{(1)}==0,\tau_2^{(1)}==0,\tau_3^{(1)}==0\},\{a_3,a_7,a_2\},\texttt{Reals]}
\end{equation}
gives four real solutions, one of which is
\begin{equation}\begin{array}{l}
a_3^*=-1.6657772441340260275382933375282366149192053853441\cdots,\\
a_7^*=-4.1967363822359183641439548974671095385907064222453\cdots,\\
a_2^*=-2.1822951200460674220108297255909871925482447960032\cdots.
\end{array}\end{equation}
For the purpose of verification, substituting $\gamma_*=(a_3^*,a_7^*,a_2^*,a_4^*)$ into \eqref{3.9} and \eqref{3.16}, we have
\begin{equation}\begin{array}{l}
\tau_1^{(1)}|_{\gamma_*}=0.\times10^{-49}\approx0,\\
\tau_2^{(1)}|_{\gamma_*}=0.\times10^{-47}\approx0,\\
\tau_3^{(1)}|_{\gamma_*}=0.\times10^{-46}\approx0,\\
\tau_4^{(1)}|_{\gamma_*}=-29.95956914787823376077394571175254368388533110\cdots\neq0,\\
\text{det}\left.\left[\frac{\partial(\tau_1,\tau_2,\tau_3)}{\partial(a_3,a_7,a_2)}\right]\right|_{\gamma_*}=-303.622391170237307523268535764275295357683185\cdots\neq0.
\end{array}\end{equation}
Theoretically speaking, the above $\tau_k^{(1)}|_{\gamma_*},k=1,2,3$ should be exactly equal to zero. However, due to numerical computation error, they are only very close to zero, which does not affect the conclusion.

\end{proof}

\subsection{Center of type $\Lambda_2$}

Substituting $\Lambda_2$ into formulae \eqref{2.15}-\eqref{2.16}, we compute the first five complex period constants at the origin of system \eqref{3.7} as follows:
\begin{equation}\begin{array}{l}
\tau_1^{(2)}=\frac{1}{12}(-48-9a_3-192a_6^2-36a_7-4a_7^2),\\
\tau_2^{(2)}=\frac{1}{54}(-972-7200a_6^2-13248a_6^4-816a_7-3264a_6^2a_7-144a_7^2-96a_6^2a_7^2+4a_7^3-135a_9-36a_7a_9),\\
\tau_3^{(2)}=\frac{1}{2570940}(12398832+49595328a_6^2+5635584a_7+22542336a_6^2a_7-420420a_7^2+8407680a_6^2a_7^2\\
\ \ \ \ \ \ \ \ \ \ \ \ \ \ \ \ \ \ \ -224640a_7^3-6920832a_6^2a_7^3-373348a_7^4-184512a_6^2a_7^4+7688a_7^5-1791558a_9\\
\ \ \ \ \ \ \ \ \ \ \ \ \ \ \ \ \ \ \ -40807152a_6^2a_9-2137239a_7a_9+15718752a_6^2a_7a_9+709731a_7^2a_9-69192a_7^3a_9\\
\ \ \ \ \ \ \ \ \ \ \ \ \ \ \ \ \ \ \ -899829a_9^2),\\
\tau_4^{(2)}=\frac{1}{6845156140890420}f_2(a_3,a_6,a_7,a_9),\\
\tau_5^{(2)}=\frac{1}{25488874431785566645482240}f_3(a_3,a_6,a_7,a_9).
\end{array}\end{equation}
In the above expression of $\tau_k^{(2)}s$, we have already let $\tau_1^{(2)}=\cdots=\tau_{k-1}^{(2)}=0,k=2,3,4,5$.

\begin{theorem}
Denote $\gamma^{(2)}=(a_3,a_6,a_7,a_9)\in\mathbb{R}^4$. For system \eqref{3.4}, if $\lambda\in\Lambda_2$ the bi-center is a weak center of at most 4. Moreover, it is a the fourth order weak center if and only if

$\Gamma_2=\{\gamma_1^{(2)},\gamma_2^{(2)},\gamma_3^{(2)},\gamma_4^{(2)},\gamma_5^{(2)},\gamma_6^{(2)},\gamma_7^{(2)},\gamma_8^{(2)},\gamma_9^{(2)},\gamma_{10}^{(2)},\gamma_{11}^{(2)},\gamma_{12}^{(2)},\gamma_{13}^{(2)},\gamma_{14}^{(2)}\}$.
\end{theorem}

\begin{proof}
Performing the \textsc{Mathematica} routine
\begin{equation}
\texttt{NSolve[}\{\tau_1^{(2)}==0,\tau_2^{(2)}==0,\tau_3^{(2)}==0,\tau_4^{(2)}==0\},\{a_3,a_6,a_7,a_9\},\texttt{Reals]}
\end{equation}
gives fourteen real solutions $\gamma_k^{(2)},k=1,2,\cdots,14$. A direct computation gives
\begin{equation}
\tau_5^{(2)}|_{\gamma_1^{(2)}}=-0.009131827261973434519460359182840789927726\neq0.
\end{equation}
The remaining thirteen cases are analogous.
\end{proof}

\begin{theorem}
For system \eqref{3.4}, if $\lambda\in\Lambda_2$, the maximum number of bifurcation of critical periods is 4, and there are exactly 4 bifurcation of critical periods after a suitable perturbation.
\end{theorem}

\begin{proof}
Checking the Jacobian matrix for this case, we see that
\begin{equation}
\text{det}\left.\left[\frac{\partial(\tau_1^{(2)},\tau_2^{(2)},\tau_3^{(2)},\tau_4^{(2)})}{\partial(a_3,a_6,a_7,a_9)}\right]\right|_{\gamma_1^{(2)}}=-0.36143871856336911594092041376646558130041029\cdots\neq0.
\end{equation}
The remaining thirteen cases are analogous.
\end{proof}

\begin{remark}
In this subsection, the reason why the maximum weak bi-center condition is expressed as a set with numerical elements is that there are finite parameter vectors in it. In contrast, in the above subsection, the maximum weak bi-center condition is expressed as a set with symbolic expressions, because there are infinite parameter vectors in it. Similar cases will occur in the next two subsections.
\end{remark}

\subsection{Center of type $\Lambda_3$}

Substituting $\Lambda_3$ into formulae \eqref{2.15}-\eqref{2.16}, we compute the first five complex period constants at the origin of system \eqref{3.7} as follows:
\begin{equation}\begin{array}{l}
\tau_1^{(3)}=\frac{1}{12}(-48-10a_2^2-9a_3+4a_2a_6-16a_6^2-36a_7-4a_7^2),\\
\tau_2^{(3)}=\frac{1}{54}(-612+420a_2^2+70a_2^4+3360a_2a_6+490a_2^3a_6+960a_6^2+1740a_2^2a_6^2+8320a_2a_6^3\\
\ \ \ \ \ \ \ \ \ \ \ \ \ +5632a_6^4-810a_7+315a_2^2a_7+2520a_2a_6a_7+720a_6^2a_7-408a_7^2+35a_2^2a_7^2+280a_2a_6a_7^2\\
\ \ \ \ \ \ \ \ \ \ \ \ \ +80a_6^2a_7^2-90a_7^3-8a_7^4),\\
\tau_3^{(3)}=\frac{1}{51840}f_4(a_3,a_7,a_2,a_6),\\
\tau_4^{(3)}=\frac{1}{2150016634828800}f_5(a_3,a_7,a_2,a_6),\\
\tau_5^{(3)}=\frac{1}{4482512233510082494464000}f_6(a_3,a_7,a_2,a_6).
\end{array}\end{equation}
In the above expression of $\tau_k^{(3)}$s, we have already let $\tau_1^{(3)}=\cdots=\tau_{k-1}^{(3)}=0,k=2,3,4,5$.

\begin{theorem}
Denote $\gamma^{(3)}=(a_3,a_7,a_2,a_6)\in\mathbb{R}^4$. For system \eqref{3.4}, if $\lambda\in\Lambda_3$ the bi-center is a weak center of at most 4. Moreover, it is a the fourth order weak center if and only if

$\Gamma_3=\{\gamma_1^{(3)},\gamma_2^{(3)},\gamma_3^{(3)},\gamma_4^{(3)},\gamma_5^{(3)},\gamma_6^{(3)},\gamma_7^{(3)},\gamma_8^{(3)},\gamma_9^{(3)},\gamma_{10}^{(3)},\gamma_{11}^{(3)},\gamma_{12}^{(3)},\gamma_{13}^{(3)},\gamma_{14}^{(3)},\gamma_{15}^{(3)},\gamma_{16}^{(3)}\}$.
\end{theorem}

\begin{proof}
Performing the \textsc{Mathematica} routine
\begin{equation}
\texttt{NSolve[}\{\tau_1^{(3)}==0,\tau_2^{(3)}==0,\tau_3^{(3)}==0,\tau_4^{(3)}==0\},\{a_3,a_7,a_2,a_6\},\texttt{Reals]}
\end{equation}
gives the sixteen real solutions $\gamma_k^{(3)},k=1,2,\cdots,16$. A direct computation gives
\begin{equation}
\tau_5^{(3)}|_{\gamma_1^{(3)}}=-455.788157320391380470485440367876484572373\neq0.
\end{equation}
The remaining fifteen cases are analogous.
\end{proof}

\begin{theorem}
For system \eqref{3.4}, if $\lambda\in\Lambda_3$, the maximum number of bifurcation of critical periods is 4, and there are exactly 4 bifurcation of critical periods after a suitable perturbation.
\end{theorem}

\begin{proof}
Checking the Jacobian matrix for this case, we see that
\begin{equation}
\text{det}\left.\left[\frac{\partial(\tau_1^{(3)},\tau_2^{(3)},\tau_3^{(3)},\tau_4^{(3)})}{\partial(a_3,a_7,a_2,a_6)}\right]\right|_{\gamma_1^{(3)}}=793811.91919002470497313053176662032132480\cdots\neq0.
\end{equation}
The remaining fifteen cases are analogous.
\end{proof}

\subsection{Center of type $\Lambda_4$}

Substituting $\Lambda_4$ into formulae \eqref{2.15}-\eqref{2.16}, we compute the first four complex period constants at the origin of system \eqref{3.7} as follows:
\begin{equation}\begin{array}{l}
\tau_1^{(4)}=\frac{1}{12}(-16-10a_2^2-9a_3+4a_2a_6-16a_6^2),\\
\tau_2^{(4)}=\frac{1}{108}(-256+280a_2^2+140a_2^4-189a_2a_4+896a_2a_6+140a_2^3a_6-162a_4a_6-512a_6^2\\
\ \ \ \ \ \ \ \ \ \ \ \ \ \ +72a_2^2a_6^2+896a_2a_6^3-256a_6^4),\\
\tau_3^{(4)}=\frac{1}{25920}(32768+6259328a_2^2-7443240a_2^4-3637620a_2^6+226800a_2a_4+4896927a_2^3a_4\\
\ \ \ \ \ \ \ \ \ \ \ \ \ \ \ \ -20412a_4^2+993280a_2a_6-25316928a_2^3a_6-4553220a_2^5a_6+284256a_4a_6\\
\ \ \ \ \ \ \ \ \ \ \ \ \ \ \ \ +5352966a_2^2a_4a_6+32768a_6^2+6009216a_2^2a_6^2-3151896a_2^4a_6^2+1432512a_2a_4a_6^2\\
\ \ \ \ \ \ \ \ \ \ \ \ \ \ \ \ +1001472a_2a_6^3-24025728a_2^3a_6^3+495360a_4a_6^3),\\
\tau_4^{(4)}=\frac{1}{202078616107242240}f_7(a_3,a_6,a_2,a_4).
\end{array}\end{equation}
In the above expression of $\tau_k^{(4)}$s, we have already let $\tau_1^{(4)}=\cdots=\tau_{k-1}^{(4)}=0,k=2,3,4$.

\begin{theorem}
Denote $\gamma^{(4)}=(a_3,a_6,a_2,a_4)\in\mathbb{R}^4$. For system \eqref{3.4}, if $\lambda\in\Lambda_4$ the bi-center is a weak center of at most 3. Moreover, it is a the third order weak center if and only if

$\Gamma_4=\left\{\gamma^{(4)}\in\mathbb{R}^4\left|\tau_1^{(4)}=\tau_2^{(4)}=\tau_3^{(4)}=0\right.\right\}$.
\end{theorem}

\begin{proof}
First of all, computing the resultant of $\tau_1^{(4)},\tau_2^{(4)},\tau_3^{(4)},\tau_4^{(4)}$ with respect to $a_3$, we have
\begin{equation}\begin{array}{l}
R_{1,2}^{(4)}=\texttt{Resultant[}\tau_1^{(4)},\tau_2^{(4)},a_3\texttt{]}=\frac{1}{108}F_{1,2}^{(4)}(a_6,a_2,a_4),\\
R_{1,3}^{(4)}=\texttt{Resultant[}\tau_1^{(4)},\tau_3^{(4)},a_3\texttt{]}=\frac{1}{25920}F_{1,3}^{(4)}(a_6,a_2,a_4),\\
R_{1,4}^{(4)}=\texttt{Resultant[}\tau_1^{(4)},\tau_4^{(4)},a_3\texttt{]}=\frac{1}{202078616107242240}F_{1,4}^{(4)}(a_6,a_2,a_4).
\end{array}\end{equation}
Next, computing the resultant of $R_{1,2}^{(4)},R_{1,3}^{(4)},R_{1,4}^{(4)}$ with respect to $a_6$, we have
\begin{equation}\begin{array}{l}
R_{12,13}^{(4)}=\texttt{Resultant[}R_{1,2}^{(4)},R_{1,3}^{(4)},a_6\texttt{]}=\frac{1}{185961834720000}F_{12,13}^{(4)}(a_2,a_4),\\
R_{12,14}^{(4)}=\texttt{Resultant[}R_{1,2}^{(4)},R_{1,4}^{(4)},a_6\texttt{]}=\frac{F_{12,14}^{(4)}(a_2,a_4)}{4006662718879818705145403830481966659700018498572980313175832330240000}.
\end{array}\end{equation}
Finally, computing the resultant of $R_{12,13}^{(4)},R_{12,14}^{(4)}$ with respect to $a_2$, we have
\begin{equation}\begin{array}{l}
R_{1213,1214}^{(4)}=\texttt{Resultant[}R_{12,13}^{(4)},R_{12,14}^{(4)},a_2\texttt{]}=\frac{7730993719707444524137094407a_4^{26}(1+16a_4^2)F_{1213,1214}^{(4)}(a_4)G_{1213,1214}^{(4)}(a_4)}{N_{1663}},
\end{array}\end{equation}
where $F_{1213,1214}^{(4)}(a_4)$ and $G_{1213,1214}^{(4)}(a_4)$ are unary polynomial in $a_4$ of degree $46$ and $182$, respectively.

Performing the \textsc{Mathematica} routine $\texttt{NSolve[}R_{1213,1214}^{(4)}==0,a_4\texttt{]}$ only gives one real solution $a_4=0$, under which we have
\begin{equation}\label{3.51}\begin{array}{l}
R_{12,13}^{(4)}|_{a_4=0}=-\frac{1}{1452826833750}a_2^6(1+a_2^2)(16+9a_2^2)(2199023255552+3148960642367488a_2^2\\
\ \ \ \ \ \ \ \ \ \ \ \ \ \ \ \ \ \ \ \ \ \ \ \ \ \ \ \ \ \ \ \ \ \ \ \ +48607541078261760a_2^4+264426711058227200a_2^6\\
\ \ \ \ \ \ \ \ \ \ \ \ \ \ \ \ \ \ \ \ \ \ \ \ \ \ \ \ \ \ \ \ \ \ \ \ +692618123639590400a_2^8+983111352623484000a_2^{10}\\
\ \ \ \ \ \ \ \ \ \ \ \ \ \ \ \ \ \ \ \ \ \ \ \ \ \ \ \ \ \ \ \ \ \ \ \ +784228248389171250a_2^{12}+294029898225170625a_2^{14}),
\end{array}\end{equation}
which has only one real root $a_2=0$. Then we compute
\begin{equation}\label{3.52}
\tau_2^{(4)}|_{a_2=a_4=0}=-\frac{64}{27}(1+a_6^2)^2\neq0,
\end{equation}
which means $\tau_1^{(4)},\tau_2^{(4)},\tau_3^{(4)},\tau_4^{(4)}$ can not be simultaneous zero. Therefore, the bi-center of system \eqref{3.4} is at most the third order weak center.
\end{proof}

\begin{theorem}
For system \eqref{3.4}, if $\lambda\in\Lambda_4$, the maximum number of bifurcation of critical periods is 3, and there are exactly 3 bifurcation of critical periods after a suitable perturbation.
\end{theorem}

\begin{proof}
The process is similar to the proof of the former theorem. The determinant of Jacobian matrix in this case is computed as
\begin{equation}
\text{det}\left[\frac{\partial(\tau_1^{(4)},\tau_2^{(4)},\tau_3^{(4)})}{\partial(a_3,a_6,a_2)}\right]=\frac{1}{311040}G(a_6,a_2,a_4).
\end{equation}
First of all, computing the resultant of $\tau_1^{(4)},G(a_6,a_2,a_4)$ with respect to $a_3$, we have
\begin{equation}
S_{1,4}=\texttt{Resultant[}\tau_1^{(4)},G(a_6,a_2,a_4),a_3\texttt{]}=\frac{1}{311040}G_{1,4}(a_6,a_2,a_4).
\end{equation}
Next, computing the resultant of $R_{1,2}^{(4)},S_{1,4}$ with respect to $a_6$, we have
\begin{equation}
S_{12,14}=\texttt{Resultant[}R_{1,2}^{(4)},S_{1,4},a_6\texttt{]}=-\frac{4}{1158137618032400625}G_{12,14}(a_2,a_4).
\end{equation}
Finally, computing the resultant of $R_{12,13}^{(4)},S_{12,14}$ with respect to $a_2$, we have
\begin{equation}
S_{1213,1214}=\texttt{Resultant[}R_{12,13}^{(4)},S_{12,14},a_2\texttt{]}=\frac{N_{51}}{N_{733}}a_4^{12}P_{1213,1214}(a_4)Q_{1213,1214}(a_4),
\end{equation}
where $P_{1213,1214}(a_4)$ and $Q_{1213,1214}(a_4)$ are unary polynomial in $a_4$ of degree 60 and 184, respectively.

Performing the \textsc{Mathematica} routine $\texttt{NSolve[}S_{1213,1214}==0,a_4\texttt{]}$ only gives one real solution $a_4=0$, under which we again get \eqref{3.51} and further \eqref{3.52}. Thus, we have
\begin{equation}
\text{det}\left.\left[\frac{\partial(\tau_1^{(4)},\tau_2^{(4)},\tau_3^{(4)})}{\partial(a_3,a_6,a_2)}\right]\right|_{\Gamma_4}\neq0.
\end{equation}
\end{proof}

To conclude, the main result is as follows:

\begin{theorem}
For system \eqref{3.4}, the bi-center is a weak center of order 3 (resp. 4) if and only if $\lambda\in\Lambda_1^{(3)}\cup\Lambda_2^{(3)}$ (resp. $\lambda\in\Lambda_1^{(4)}\cup\Lambda_2^{(4)}$), and there are exactly 3 (resp. 4) bifurcation of critical periods after a suitable perturbation, where

$\Lambda_1^{(3)}=\left\{\lambda\in\mathbb{R}^8\left|\lambda\in\Lambda_1,\gamma^{(1)}\in\Gamma_1\right.\right\}$;

$\Lambda_2^{(3)}=\left\{\lambda\in\mathbb{R}^8\left|\lambda\in\Lambda_4,\gamma^{(4)}\in\Gamma_4\right.\right\}$;

$\Lambda_1^{(4)}=\left\{\lambda\in\mathbb{R}^8\left|\lambda\in\Lambda_2,\gamma^{(2)}\in\Gamma_2\right.\right\}$;

$\Lambda_2^{(4)}=\left\{\lambda\in\mathbb{R}^8\left|\lambda\in\Lambda_3,\gamma^{(3)}\in\Gamma_3\right.\right\}$.
\end{theorem}

\section{Conclusion}

In this paper, we address the problem of the order of weak centers and maximum number of critical periods bifurcating
from the bi-center $(\pm1,0)$ in system \eqref{3.4} as the parameter varies and solve it completely. We have the following:
\begin{itemize}
\item The bi-center is of finite order.

\item The result obtained in [Romanovski et al, 2017] that the family has no isochronous bi-center is verified.

\item At most four critical periods can bifurcate from the bi-center.

\item To identify to maximum order of weak bi-center and give the maximum number of critical periods, we reduce the question to that of transversal intersections of some algebraic surfaces, using computational algebraic geometry techniques such as the Theory of Resultant and Gr\"{o}bner Bases.

\end{itemize}

Throughout the paper some polynomials $f_i,i=1,2,3,4,5,6,7$ are not given explicit expressions, because they are very large. In fact, the readers can easily compute them using any available computer algebra system. However, these polynomials are available via the following E-mail address: \textsf{wuyusen621@126.com}.

\section*{Acknowledgements}

The author of this research article express his profound gratitude and sincerest thanks to the anonymous reviewers, for their constructive opinions towards making this research article scientifically sound.

\section*{Appendix}

The process of computing and simplifying $\tau_1^{(1)}$ and $\tau_2^{(1)}$ in \eqref{3.9}:

Substituting \eqref{3.8} into \eqref{2.15} and \eqref{2.16}, we get the recursive formulae of $c'[k,j],d'[k,j],p'[j],q'[j]$ and $\tau[j]=p'[j]+q'[j]$, the \textsc{Mathematica} codes are as follows:
\begin{figure}[!hbt]
\flushleft
\includegraphics[height=1.3cm,width=3.8cm]{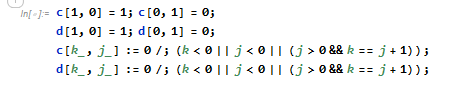}
\end{figure}

\begin{figure}[!ht]
\centering
\includegraphics[height=6.6cm,width=18cm]{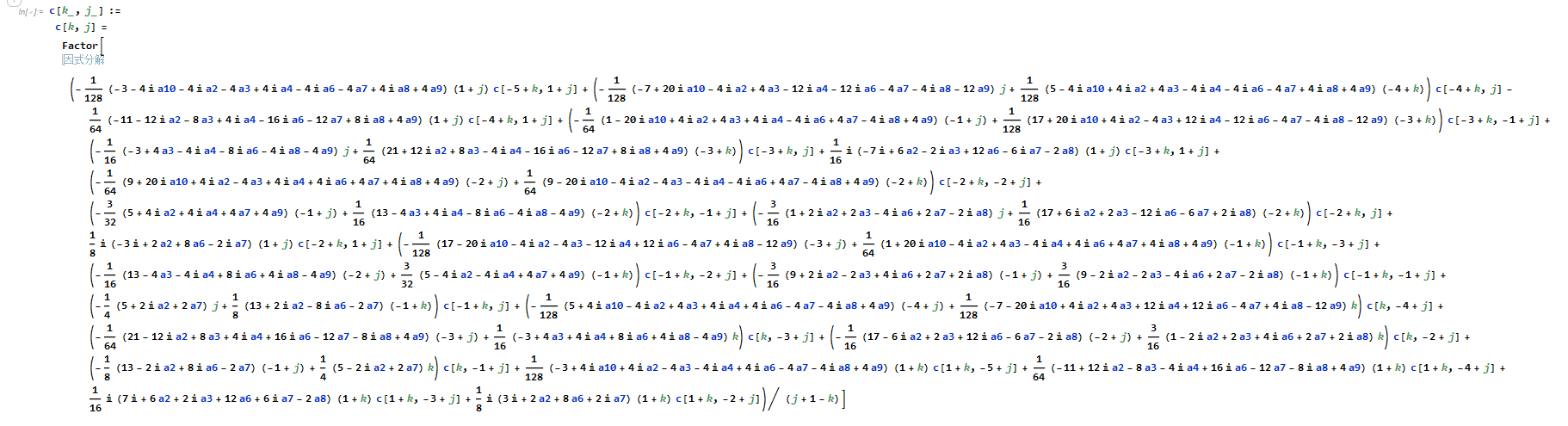}
\end{figure}

\begin{figure}[!ht]
\centering
\includegraphics[height=6.6cm,width=18cm]{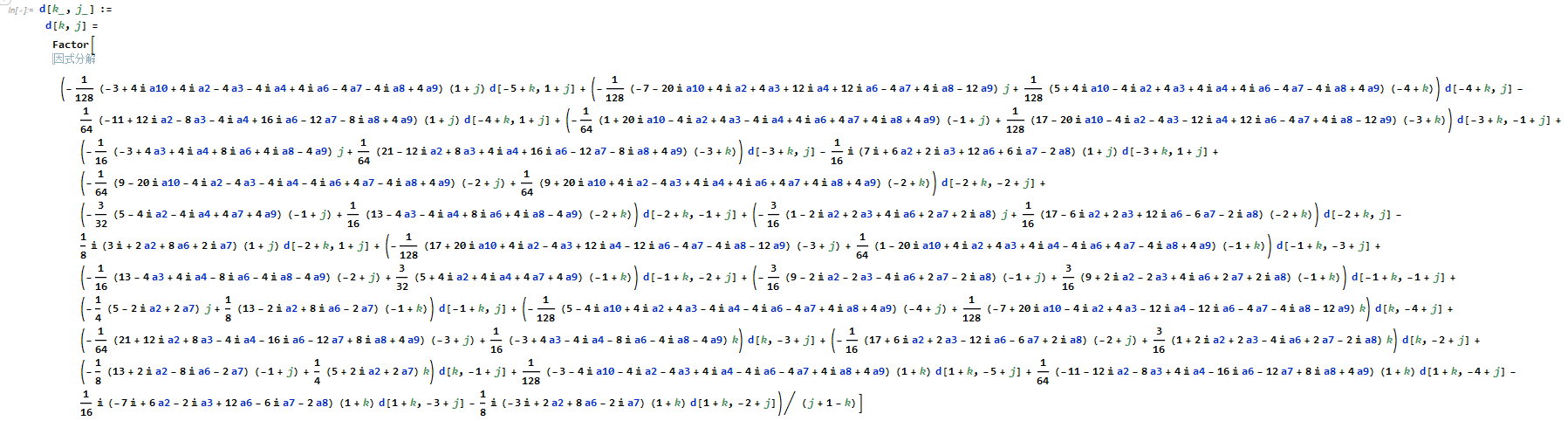}
\end{figure}

\newpage

\begin{figure}[!ht]
\centering
\includegraphics[height=12cm,width=18cm]{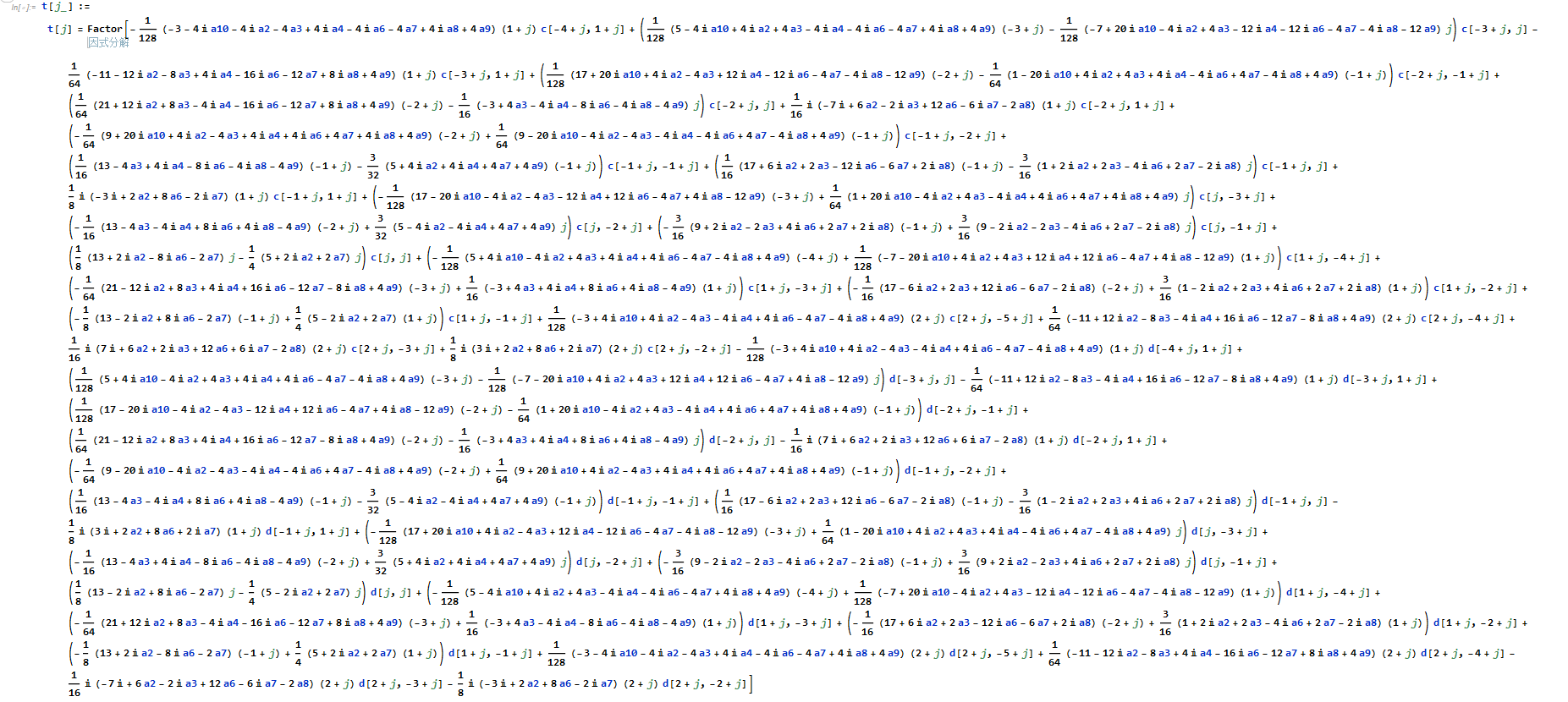}
\end{figure}

Setting $a_6=0,a_8=\frac{1}{3}a_2(1-2a_7),a_9=\frac{1}{2}a_3(1-a_7),a_{10}=\frac{1}{5}a_4(3-2a_7)$ and executing the foregoing \textsc{Mathematica} codes, we have

$\tau_1^{(1)}=\frac{1}{12}(-48-10a_2^2-9a_3-36a_7-4a_7^2),$

$\begin{array}{l}
\tau_2^{(1)}=\frac{1}{384}(-5184-824a_2^2-260a_2^4-972a_3-852a_2^2a_3-153a_3^2-672a_2a_4-2160a_7+392a_2^2a_7\\
\ \ \ \ \ \ \ \ \ \ \ \ \ \ +180a_3a_7+3376a_7^2+608a_2^2a_7^2+528a_3a_7^2+2096a_7^3+208a_7^4),
\end{array}$

\noindent Let

$k_{21}=\frac{1}{288}(156+682a_2^2+153a_3-792a_7-596a_7^2),$

\noindent and

$\tau_2^{(1)}\to\tau_2^{(1)}-k_{2,1}\tau_1^{(1)},$

\noindent then

$\tau_2^{(1)}=\frac{1}{108}(-1224+840a_2^2+140a_2^4-189a_2a_4-1620a_7+630a_2^2a_7-816a_7^2+70a_2^2a_7^2-180a_7^3-16a_7^4).$

%\end{multicols}
\end{document}